\documentclass[reqno]{amsart}
\usepackage{amsthm,amsmath,amssymb}
\usepackage{stmaryrd,mathrsfs}
\usepackage[T1]{fontenc}
\usepackage{esint}
\usepackage{tikz}

\allowdisplaybreaks

\newcommand{\ud}[0]{\,\mathrm{d}}

\newcommand{\abs}[1]{|#1|}
\newcommand{\Babs}[1]{\Big|#1\Big|}

\newcommand{\Norm}[2]{\|#1\|_{#2}}

\newcommand{\BNorm}[2]{\Big\|#1\Big\|_{#2}}
\newcommand{\pair}[2]{\langle #1,#2 \rangle}
\newcommand{\Bpair}[2]{\Big\langle #1,#2 \Big\rangle}
\newcommand{\ave}[1]{\langle #1\rangle}

\newcommand{\bddlin}[0]{\mathscr{L}}

\newcommand{\BMO}[0]{\operatorname{BMO}}
\newcommand{\supp}[0]{\operatorname{supp}}
\newcommand{\loc}[0]{\operatorname{loc}}

\newcommand{\R}{\mathbb{R}}
\newcommand{\C}{\mathbb{C}}

\newcommand{\Z}{\mathbb{Z}}

\renewcommand{\P}[0]{\mathbb{P}}
\newcommand{\E}[0]{\mathbb{E}}
\newcommand{\D}[0]{\mathbb{D}}
\newcommand{\eps}[0]{\varepsilon}

\swapnumbers \numberwithin{equation}{section}

\theoremstyle{plain}
\newtheorem{theorem}[equation]{Theorem}

\theoremstyle{definition}

\theoremstyle{remark}
\newtheorem{remark}[equation]{Remark}

\makeatletter
\@namedef{subjclassname@2010}{%
  \textup{2010} Mathematics Subject Classification}
\makeatother

\begin{document}

\title[A symmetric $T(1)$ theorem]{An operator-valued $T(1)$ theorem for symmetric singular integrals in UMD spaces}

\author{Tuomas Hyt\"onen}
\address{Department of Mathematics and Statistics, P.O.B.~68 (Pietari Kalmin katu 5), FI-00014 University of Helsinki, Finland}
\email{tuomas.hytonen@helsinki.fi}

%\date{\today}

\thanks{The author was supported by the Academy of Finland through project Nos. 307333 (Centre of Excellence in Analysis and Dynamics Research) and 314829 (Frontiers of singular integrals).}
% 42B20 Singular and oscillatory integrals (Calder\'on-Zygmund, etc.) 
% 42B25 Maximal functions, Littlewood-Paley theory
% 42B35 Function spaces arising in harmonic analysis 

% 46B09 Probabilistic methods in Banach space theory
% 46E40 Spaces of vector- and operator-valued functions
% 47A60 Functional calculus
% 47F05 Partial differential operators
% 60G46 Martingales and classical analysis

\maketitle

%\vspace*{-2mm}

\begin{abstract}
The natural BMO (bounded mean oscillation) conditions suggested by scalar-valued results are known to be insufficient for the boundedness of operator-valued paraproducts. Accordingly, the boundedness of operator-valued singular integrals has only been available under
versions of the classical ``$T(1)\in\BMO$'' assumptions that are not easily checkable.
Recently, Hong, Liu and Mei (J. Funct. Anal. 2020) observed that the situation improves remarkably for singular integrals with a symmetry assumption, so that a classical $T(1)$ criterion still guarantees their $L^2$-boundedness on Hilbert space -valued functions. Here, these results are extended to general UMD (unconditional martingale differences) spaces with the same natural BMO condition for symmetrised paraproducts, and requiring in addition only the usual replacement of uniform bounds by $R$-bounds in the case of general singular integrals. In particular, under these assumptions, we obtain boundedness results on non-commutative $L^p$ spaces for all $1<p<\infty$, without the need to replace the domain or the target by a related non-commutative Hardy space as in the results of Hong et al. for $p\neq 2$.
\end{abstract}

%\tableofcontents

\section{Introduction}

Paraproducts are central and widespread in modern harmonic analysis. Their systematic introduction is due to J.-M. Bony \cite{Bony:81} in the context of symbolic calculus for non-linear partial differential equations but, as argued by \cite{BMV}, ``the first version of a paraproduct is [already] implicit in A. P. Calder\'on's work on commutators \cite{Calderon:65}''. We refer the reader to \cite{BMV} for a friendly introduction to the variety of objects that now go under the generic name ``paraproduct'', and some further applications. Our present interest is in the  {\em dyadic} version of paraproducts $\Pi_b$, and in their role in the boundedness of general singular integral operators, a connection revealed by the celebrated $T(1)$ and $T(b)$ theorems of G.~David, J.-L. Journ\'e and S. Semmes \cite{DJ:T1,DJS:Tb}.

More specifically, we are interested in the matrix/operator-valued versions of these objects and results. Matrix-valued paraproducts (disguised as matrix-valued Carleson embeddings) appeared in the work of S. Treil and A. Volberg \cite{TV:angle}, motivated by questions in multivariate stationary processes. Independently, these authors with F. Nazarov  \cite{NTV:ex} on the one hand, and N. H. Katz \cite{Katz:parap} on the other hand, obtained the dimensional dependence in the key inequality
\begin{equation}\label{eq:logParap}
  \Norm{b}{\BMO^d_{\rm so}(\R;\bddlin(H))}\leq\Norm{\Pi_b}{\bddlin(L^2(\R;H))} \leq c(1+\log \dim H)\Norm{b}{\BMO^d_{\rm so}(\R;\bddlin(H))}, 
\end{equation}
where $H$ is a finite-dimensional Hilbert space, $\bddlin(H)$ is the space of bounded linear operators acting on this space, $\Pi_b$ is the dyadic paraproduct associated with the $\bddlin(H)$-valued function $b$, and
\begin{equation}\label{eq:BMOso}
  \Norm{b}{\BMO^d_{\rm so}(\R;\bddlin(H))} 
  :=\sup_{\substack{x\in H \\ \Norm{x}{}\leq 1}}\sup_{\substack{ I\subset\R \\ \textup{interval}}}\frac{1}{\abs{I}}\int_I\Norm{(b(t)-\ave{b}_I)x}{H}\ud t
\end{equation}
is the strong-operator (dyadic) bounded mean oscillation norm. Some time later, Nazarov, Pisier, Treil and Volberg \cite{NPTV} proved the sharpness of $\log \dim H$ in \eqref{eq:logParap}; the necessity of some dimensional growth was already contained in \cite{NTV:ex}. (In \eqref{eq:BMOso}, we have followed the notation ``$\BMO_{\rm so}$'' as used e.g. in \cite{NPTV}, but this is not universal; some other related papers also incorporate the norm of the pointwise adjoint function $t\mapsto b(t)^*$ in this notation.)

For a while, there were hopes in the air of achieving a dimension-free bound by replacing $\BMO_{\rm so}$ on the right of \eqref{eq:logParap} by the larger uniform-operator BMO norm
\begin{equation*}
   \Norm{b}{\BMO_d(\R;\bddlin(H))}
   := \sup_{\substack{ I\subset\R \\ \textup{interval}}}\frac{1}{\abs{I}}\int_I\Norm{b(t)-\ave{b}_I}{\bddlin(H)}\ud t
\end{equation*}
but also this was ruled out by T. Mei \cite{Mei:parap} by showing that, even in the bound
\begin{equation*}
  \Norm{\Pi_b}{\bddlin(L^2(\R;H))} \leq \phi(\dim H)\Norm{b}{L^\infty(\R;\bddlin(H))},
\end{equation*}
with the smaller function space $L^\infty$ in place of $\BMO_d$ on the right, the dimensional dependence cannot be better than $\phi(d)\geq c(1+\log d)$.

In particular, this severely sets back the hopes of developing such estimates in infinite-dimensional Hilbert spaces, not to mention more general Banach spaces. This failure in infinite dimensions of classical bounds between various BMO-type quantities on the one hand, and the norms of related transformations on the other hand, has been further elaborated by Blasco and Pott \cite{BlaPo:08,BlaPo:10} and, for analogous questions dealing with the (complex-)analytic BMOA and related operators, quite recently by Rydhe \cite{Rydhe:17}.

\smallskip

As mentioned, one of the major applications of paraproducts is their role in the characterisation of boundedness of general (non-convolution type) singular integral operators via the $T(1)$ and $T(b)$ theorems of David, Journ\'e and Semmes \cite{DJ:T1,DJS:Tb}, as well as their many extensions. The above-discussed problems of describing the boundedness of paraproducts in the infinite-dimensional setting, in terms of accessible function space norms, have been a major obstacle on the way of achieving a fully satisfactory analogue of the general theory of singular integrals in infinite-dimensional Banach spaces; the available versions of the operator-valued $T(1)$ and $T(b)$ theorems  -- \cite{HH:2016,Hytonen:Tb,HW:2006} and their extensions -- suffer from complicated and not easily verifiable variants of BMO conditions that are only distant cousins of their simple classical predecessors. 

In contrast, verifiable conditions for the boundedness of operator-valued singular integrals of {\em convolution type},
\begin{equation*}
  Tf(x)=\int_{\R^d}K(x-y)f(y)\ud y,
\end{equation*}
are well understood since the work of L.~Weis \cite{Weis:2001} via their equivalent description as operator-valued Fourier multipliers $\widehat{Tf}=\hat{K}\hat{f}$ (but see also \cite{HW:2007} for a singular integral point of view to the same operators). Results on the boundedness of these operator-valued convolution-type singular integrals have profound applications to regularity problems for autonomous evolution equations; see again \cite{Weis:2001} and the many works citing this influential paper.

Analogous questions for non-autonomous equations give rise to singular integrals of non-convolution type,
\begin{equation*}
  Tf(x)=\int_{\R^d}K(x,y)f(y)\ud y,
\end{equation*}
which in principle should belong to the scope of the (operator-valued) $T(1)$ and $T(b)$ theorems \cite{HH:2016,Hytonen:Tb,HW:2006}. However, the complicated form of their conditions has so far hindered such applications. In a recent paper  \cite[p. 535]{GV:2017}, the authors explicitly write:  ``At the moment we do not know whether the $T1$-theorem and $Tb$-theorem can be applied to study maximal $L^p$-regularity for the time dependent problems we consider.''
Nevertheless, the authors of \cite{GV:2017} manage to obtain the $L^p$-boundedness for a special class of operator-valued non-convolution singular integrals suitable for their needs. This indicates a continuing demand for checkable criteria for the boundedness of at least special classes of non-convolution operators, as long as the full analogue of the scalar-valued $T(1)$ and $T(b)$ theorems seems out of reach.

In a recent work \cite{HLM}, Hong, Liu and Mei achieve such a result, in the very style of a $T(1)$ theorem, for operator-valued singular integrals with a certain {\em symmetry} assumption, satisfied in particular by all even operators. Under natural assumptions, their result gives the $L^p(\R^d;X)$-boundedness of these operators when $p=2$ and $X=H$ is a Hilbert space, and a weaker substitute result (replacing either the domain or the target with a suitable non-commutative Hardy space) when $p\in(1,\infty)\setminus\{2\}$ and $X$ is non-commutative $L^p$-space. (Incidentally, a symmetry condition was also key to another recent advance on vector-valued singular integrals concerning the possible linear dependence of singular integral and martingale transform norms, a problem that was solved for even singular integrals by Pott and Stoica \cite{PS:2014} but remains open in general.)

In this paper, we obtain an extension of the Hong--Liu--Mei \cite{HLM} result to all Banach spaces $X$ with the unconditionality property of martingale differences (UMD). Our result is a pure $L^p$ estimate for all $p\in(1,\infty)$, without the need of substitute Hardy spaces, and it is obtained in the maximal generality of Banach spaces (namely, UMD spaces) in which such results could be hoped for. Indeed, the Beurling--Ahlfors transform 
\begin{equation*}
  Tf(z)=-\operatorname{p.v.}\frac{1}{\pi}\int_{\C}\frac{f(y)}{(z-y)^2}\ud y,
\end{equation*}
where the integration is with respect to the two-dimensional Lebesgue measure on $\C\eqsim\R^2$, is an even singular integral operator in the scope of Theorem \ref{thm:T1intro} below, whose boundedness on $L^p(\R^2;X)$ is {\em equivalent} to $X$ being a UMD space by \cite{GMS}.

Our main result can be roughly stated as follows; 
see Theorem \ref{thm:TbSymm} for a detailed formulation of the various technical assumptions appearing in the statement.

\begin{theorem}[Symmetric $T(1)$ theorem]\label{thm:T1intro}
Let $X$ be a UMD space, $\bddlin(X)$ the space of bounded linear operators on $X$, and $p\in(1,\infty)$. Let
\begin{equation*}
  Tf(x)=\int_{\R^d}K(x,y)f(y)\ud y
\end{equation*}
be a Calder\'on--Zygmund operator with an $\bddlin(X)$-valued kernel $K$, acting on $X$-valued test functions $f$. Suppose that $T$ and $K$ satisfy $R$-bounded versions of the Calder\'on--Zygmund kernel estimates and the weak boundedness property as defined in Section \ref{sec:T1setup}.
Finally, suppose that
\begin{equation}\label{eq:T1symm}
  T1=(T^*1)^*\in\BMO(\R^d;\bddlin(X)),
\end{equation}
Then $T$ extends to a bounded linear operator on $L^p(\R^d;X)$.
\end{theorem}

We stress that all other assumptions of Theorem \ref{thm:T1intro} are essentially the same as in any other operator-valued $T(1)$ theorem in the literature (like \cite{HH:2016,Hytonen:Tb,HW:2006}), and the key novelty is the condition \eqref{eq:T1symm} inspired by \cite{HLM}. This improves on all previous results on the level of general UMD spaces by means of replacing their more complicated BMO-type spaces by the plain $\BMO(\R^d;\bddlin(X))$, which is just the classical BMO space with absolute values replaced the norm in $\bddlin(X)$; it achieves this at the cost of requiring the additional symmetry imposed by the equality in \eqref{eq:T1symm}. Unfortunately, Theorem \ref{thm:T1intro} still lacks a key feature of the classical $T(1)$ theorems: a {\em characterisation} of $L^p$-boundedness. The condition that $T1\in\BMO(\R^d;\bddlin(X))$ is a relatively checkable {\em sufficient} condition, but it is still {\em not necessary}.

As with the other assumptions, we refer the reader to Section \ref{sec:T1setup} for a precise interpretation of the condition \eqref{eq:T1symm}; however, the following formal explanation may be helpful at this point: In \eqref{eq:T1symm}, the (formal) action of $T$ on the constant scalar function $1$ is an $\bddlin(X,Y)$-valued function, and part of the assumption \eqref{eq:T1symm} is to require that this function has bounded mean oscillation. Moreover, $T^*$ refers to the formal adjoint operator
\begin{equation*}
  T^* g(x)=\int_{\R^d} K(y,x)^* g(y)\ud y,
\end{equation*}
whose kernel $K(y,x)^*$ takes values in $\bddlin(X^*)$ and acts on test functions $g$ with values in $X^*$. The formal action of $T^*$ on the constant scalar function $1$ is an $\bddlin(X^*)$-valued function $T^*1$, and $(T^*1)^*$ refers to its pointwise adjoint, an $\bddlin(X^{**})=\bddlin(X)$-valued function. (Note that UMD spaces are reflexive, see \cite[Theorem 4.3.3.]{HNVW1}.) Thus $T1$ and $(T^*1)^*$ are (at least formally) functions of the same type, and another part of the assumption \eqref{eq:T1symm} is to require that they are equal.

A main ingredient of Theorem \ref{thm:T1intro}, related to the key condition \eqref{eq:T1symm}, is the following bound of independent interest for the symmetrised paraproduct
\begin{equation*}
  \Lambda_b:=\Pi_b+\Pi_{b^*}^*;
\end{equation*}
the precise definition of these operators and a slightly more general statement will be given in Section \ref{sec:parap}.

\begin{theorem}[Boundedness of symmetrised paraproducts]\label{thm:parapIntro}
Let $X$ be a UMD space and $b\in\BMO_d(\R^d;\bddlin(X))$. Then the symmetrised paraproduct $\Lambda_b$ extends to a bounded linear operator on $L^p(\R^d;X)$ with the bound
\begin{equation*}
  \Norm{\Lambda_b}{\bddlin(L^p(\R^d;X))}
  \leq c_d\beta_{p,X}^2\Norm{b}{\BMO(\R^d;\bddlin(X))},
\end{equation*}
where $\beta_{p,X}$ is the UMD constant of $X$ and $c_d$ is dimensional.
\end{theorem}

We note that the case when $p=2$ and $X=H$ is a Hilbert space is already due to Blasco and Pott \cite[Theorem 2.6]{BlaPo:08}. Hong, Liu and Mei \cite[Proposition 2.2]{HLM} provide a version with $p\in(1,\infty)$, where $X$ is a non-commutative $L^p$ space (with the same~$p$), and either the domain or the target needs to be replaced by an appropriate non-commutative Hardy space in place of  $L^p(\R^d;X)$ when $p\neq 2$. This \cite[Proposition 2.2]{HLM} plays a similar role in their $T(1)$ theorem,  as Theorem \ref{thm:parapIntro} in our Theorem \ref{thm:T1intro}, and motivates our approach to both results.

Theorem \ref{thm:T1intro} is actually a relatively quick corollary of Theorem \ref{thm:parapIntro} and the intermediate results in essentially any existing proof of the operator-valued $T(1)$ theorem based on the dyadic approach. Namely, these proofs typically decompose the operator into a sum of two paraproducts $\Pi_{b_1}+\Pi_{b_2^*}^*$, estimated with the help of some BMO type assumptions, and the cancellative part, which is handled by using the Calder\'on--Zygmund kernel estimates and the weak boundedness property. For the cancellative part, we can simply borrow the estimates that were already carried out in one of the previous works. For the paraproduct part, under our symmetry assumption \eqref{eq:T1symm}, we have $b_1=b_2=b$, and hence this part reduces to
\begin{equation*}
  \Pi_{b_1}+\Pi_{b_2^*}^*=\Pi_b+\Pi_{b^*}^*=\Lambda_b,
\end{equation*}
which is precisely the operator estimated in Theorem \ref{thm:parapIntro}. Such a decomposition into the paraproduct part and the cancellative part is at least implicitly behind essentially all known proofs of the $T(1)$ theorem, but it is particularly clean in the recent {\em dyadic representation theorems} that originate  from the resolution of the $A_2$ conjecture on sharp weighted norm inequalities \cite{Hytonen:A2}. For our purposes, we use the {\em operator-valued} dyadic representation theorem from \cite{HH:2016}. (Hong, Liu and Mei \cite{HLM} adapt the approach of \cite{Hytonen:revisit} instead; this would also have been a relevant alternative here.)s

The rest of this paper is structured as follows. In Section \ref{sec:prelim}, we collect some necessary preliminaries on the vector-valued dyadic Hardy space and BMO on the one hand, and on projective tensor products and their duality on the other hand. The latter provide a key substitute in our considerations for some of the non-commutative tools used by \cite{HLM}. In Section \ref{sec:parap}, we provide the definitions related to paraproducts, and give the proof of Theorem \ref{thm:parapIntro}. Up to this point, all considerations are purely dyadic, and we only turn to continuous singular integrals and related objects in the remaining two sections. In Section \ref{sec:T1setup}, we provide all necessary definitions to give a precise formulation of our main Theorem \ref{thm:T1intro}. This theorem is then proved in the final Section \ref{sec:T1proof} via the operator-valued dyadic representation theorem of \cite{HH:2016}, which we recall there.

\section{Preliminaries}\label{sec:prelim}

A {\em system of dyadic cubes} in $\R^d$ is a family
\begin{equation*}
  \mathscr D:=\bigcup_{k\in\Z}\mathscr D_k
\end{equation*}
of (axes-parallel, left-closed, right-open) cubes $Q$ such that, for each $k\in\Z$,
\begin{itemize}
  \item $\mathscr D_k$ is a partition of $\R^d$ consisting of cubes $Q$ of side-length $\ell(Q)=2^{-k}$;
  \item $\mathscr D_{k+1}$ is a refinement of $\mathscr D_k$.
\end{itemize}
We consider one such dyadic system fixed for the moment. However, in our approach to the $T(1)$ theorem below, it will be important that all estimates hold uniformly with respect to the choice of the dyadic systems, as the later considerations will involve a {\em random} choice.

A dyadic system induces the averaging (or conditional expectation) operators
\begin{equation*}
  \E_k f:=\sum_{Q\in\mathscr D_k} 1_Q\ave{f}_Q,\qquad\ave{f}_Q:=\fint_Q f:=\frac{1}{\abs{Q}}\int_Q f(t)\ud t
\end{equation*}
and the martingale difference operators
\begin{equation*}
  \D_k f:=\E_{k+1}f-\E_k f,
\end{equation*}
both well defined for $f\in L^1_{\loc}(\R^d;E)$, where $E$ is any Banach space.

For (say) $f\in L^p(\R^d;E)$ with $1<p<\infty$, we have $\E_k f\to f$ as $k\to\infty$, and $\E_k f\to 0$ as $k\to-\infty$, both in the norm of $L^p(\R^d;E)$ and pointwise almost everywhere (see e.g. \cite[Theorems 3.3.2 and 3.3.5]{HNVW1}) and hence
\begin{equation}\label{eq:sumDk}
  f=\sum_{k\in\Z}\D_k f;
\end{equation}
in particular, finite truncations of sums on the right are dense in $L^p(\R^d;E)$. If $E$ is a UMD space, the convergence of \eqref{eq:sumDk} is unconditional. In particular, the modified sums $\sum_{k\in\Z}\epsilon_k\D_k f$ with $\epsilon_k\in\{-1,+1\}$ also converge, and
\begin{equation*}
  \BNorm{\sum_{k\in\Z}\epsilon_k\D_k f}{L^p(\R^d;E)}
  \leq\beta_{p,E}\BNorm{\sum_{k\in\Z}\D_k f}{L^p(\R^d;E)}=\beta_{p,E}\Norm{f}{L^p(\R^d;E)},
\end{equation*}
where $\beta_{p,E}$ is the UMD constant of $E$. See \cite[Chapter 4]{HNVW1} for these results and more on UMD spaces.

\subsection{Dyadic $H^1$ and $\BMO$}
The {\em dyadic Hardy space} $H^1_d(\R^d;E)$ is defined with the help of the {\em cancellative} dyadic maximal function
\begin{equation*}
  M_dh:=\sup_{k\in\Z}\Norm{\E_k h}{E}=\sup_{Q\in\mathscr D} 1_Q\Norm{\ave{h}_Q}{E}.
\end{equation*}
It is essential that the norm is taken outside and not inside the average. Then
\begin{equation*}
  H^1_d(\R^d;E)
  :=\Big\{h\in L^1(\R^d;E):\Norm{h}{H^1_d}:=\Norm{M_dh}{1}<\infty\Big\}.
\end{equation*}
Note that the choice of $L^1(\R^d;E)$ as the ambient space does not impose any essential restriction. Even if we only demanded that $h\in L^1_{\loc}(\R^d;E)$ (an essentially minimal condition to be able to make sense of $M_dh$) it follows from Lebesgue's differentiation theorem that $\Norm{h(\cdot)}{E}\leq M_dh$ a.e., and hence $h\in L^1(\R^d;E)$ if $M_dh\in L^1(\R^d)$.

The dyadic BMO space $\BMO_d(\R^d;F)$ (with valued in another Banach space $F$) is defined as
\begin{equation*}
  \BMO_d(\R^d;F):=\Big\{b\in L^1_{\loc}(\R^d;F):\Norm{b}{\BMO_d}:=\sup_{Q\in\mathscr D}\fint_Q\Norm{b-\ave{b}_Q}{F}<\infty\Big\}.
\end{equation*}
Both the dyadic $H^1$ and BMO are special cases of {\em martingale} $H^1$ and BMO (with respect to a regular filtration). When $F=E^*$, there is a fundamental duality between these spaces. It is essential that the following key estimate of this duality is valid for an arbitrary Banach space $E$ (see \cite[Theorem 12]{Bourgain:86}):
\begin{equation}\label{eq:H1BMOduality}
\begin{split}
  \abs{\pair{b}{h}}
  :\!&=\Babs{\lim_{N\to\infty}\int_{\R^d} \min\big\{1,\frac{N}{\Norm{b(x)}{E^*}}\big\}\pair{b(x)}{h(x)}\ud x} \\
  &\lesssim\Norm{b}{\BMO_d(\R^d;E^*)}\Norm{h}{H^1_d(\R^d;E)},
\end{split}
\end{equation}
where the implied constant depends only on the dimension $d$ (more generally, on the underlying filtration, which in our case is the dyadic filtration of $\R^d$). Note that $x\mapsto\pair{b(x)}{h(x)}$ need not be integrable under these assumptions, so that the pairing needs to be defined via such a limiting process in general. Of course, {\em if} $x\mapsto\pair{b(x)}{h(x)}$ is integrable, then dominated convergence shows that the limit is simply the integral of this function.

This identifies $\BMO_d(\R^d;E^*)$ with a subspace of $(H_d^1(\R^d;E))^*$, and it is not difficult check that this identification is isomorphic, although we only need the one-sided inequality in \eqref{eq:H1BMOduality}. (It is also known that $\BMO_d(\R^d;E^*)$ exhausts the entire dual of $H^1_d(\R^d;E)$, if and only if $E^*$ has the {\em Radon--Nikod\'ym property}, cf.~\cite{Blasco:88}, we have no need for such considerations in the present context.)

\subsection{Projective tensor product and duality}
We will also need some basic facts about the {\em projective tensor product} $E\hat\otimes F$ of Banach spaces $E$ and $F$. A reference for this material is \cite[Sec. 0.b]{Pisier:fact}. The algebraic tensor product $E\otimes F$ consists of finite sums of the form
\begin{equation*}
  v=\sum_{k=1}^K e_k\otimes f_k,\qquad e_k\in E,\ f_k\in F.
\end{equation*}
On this space, we define the norm
\begin{equation*}
  \Norm{v}{\wedge}:=\inf\sum_{k=1}^K \Norm{e_k}{E}\Norm{f_k}{F},
\end{equation*}
where the infimum runs over all expansions of $v$ of this form. Then $E\hat\otimes F$ is the completion of $E\otimes F$ with respect to this norm.

Let $\mathscr B(E\times F)$ stand for the space of bounded {\em bilinear forms} on $E\times F$. This can be identified with either of the two spaces of bounded {\em linear operators} $\mathscr L(E,F^*)$ or $\mathscr L(F,E^*)$. For any $\phi\in\mathscr B(E\times F)$ and $v=\sum_{k=1}^K e_k\otimes f_k\in E\otimes F$, the pairing
\begin{equation*}
  \pair{\phi}{v}:=\sum_{k=1}^K \phi(e_k,f_k)
\end{equation*}
is well-defined, i.e., independent of the particular representation of $v$. It is then clear that
\begin{equation*}
  \abs{\pair{\phi}{v}}\leq\Norm{\phi}{\mathscr B(E,F)}\Norm{v}{\wedge},
\end{equation*}
and hence, by continuity, $\phi$ induces an element of $(E\hat\otimes F)^*$. Conversely, given $\lambda\in(E\hat\otimes F)^*$ the formula $\phi(e,f):=\lambda(e\otimes f)$ defines a bilinear form $\phi\in\mathscr B(E\times F)$, which induces $\lambda$ in the above sense. This gives rise to the isometric identification
\begin{equation*}
   (E\hat\otimes F)^*\simeq\mathscr B(E,F)\simeq\mathscr L(E,F^*). 
\end{equation*}

In combination with the $H^1$-BMO duality, this shows that
\begin{equation*}
  \abs{\pair{b}{h}}\lesssim\Norm{b}{\BMO(\R^d;\mathscr L(E,F^*))}\Norm{h}{H^1(\R^d; E\hat\otimes F)}
\end{equation*}
for all functions $b$ and $h$ in the indicated spaces. Our main interest lies in the case when $F=Y^*$ is a dual space. Via the usual identification $Y\subseteq Y^{**}$, we have $\mathscr L(E,Y)\subseteq \mathscr L(E,Y^{**})$, and hence in particular
\begin{equation*}
  \abs{\pair{b}{h}}\lesssim\Norm{b}{\BMO(\R^d;\mathscr L(E,Y))}\Norm{h}{H^1(\R^d; E\hat\otimes Y^*)}
\end{equation*}
for all functions $b$ and $h$ in the indicated spaces. Note that even if $E$ and $Y$ are very nice spaces (as they will be in our main application), the spaces $E\hat\otimes Y^*$ and $\bddlin(E,Y)$ are not, and hence it is quite essential for our purposes that we only use the part of the $H^1$-$\BMO$ duality that is valid in general Banach spaces.

\section{Paraproducts}\label{sec:parap}

The paraproduct of two functions $b$ and $f$ is the formal series
\begin{equation*}
  \Pi_b f=\sum_{k\in\Z}\D_k b\,\E_{k-1}f
  =\sum_{k\in\Z}\D_k (b\,\E_{k-1}f).
\end{equation*}
Given two Banach space $X$ and $Y$, the individual terms of this series are well-defined for $b\in L^1_{\loc}(\R^d;\bddlin(X,Y))$ and $f\in L^p(\R^d;X)$, producing $\D_k b\,\E_{k-1}f\in L^\infty_{\loc}(\R^d;Y)$. The series can then be paired agains any $g\in L^{p'}_c(\R^d;Y^*)$ (here and below, the subscript $c$ refers to compact support) with a finitely-nonzero martingale difference expansion $g=\sum_{k\in\Z}\D_k g$ by
\begin{equation*}
  \pair{\Pi_b f}{g}=\sum_{k\in\Z}\pair{ b\,\E_{k-1}f}{\D_k g}
  =\sum_{k\in\Z}\pair{b}{\E_{k-1}f\otimes\D_k g}
  =\Bpair{b}{\sum_{k\in\Z}\E_{k-1}f\otimes\D_k g}.
\end{equation*}
All pairings involving the sum over $k\in\Z$ are the integral pairings $\pair{F}{G}=\int\pair{F(x)}{G(x)}\ud x$ where the pointwise pairing $\pair{F(x)}{G(x)}$ is between $Y$ and $Y^*$ in the first sum, and between $\bddlin(X,Y)$ and $X\otimes Y^*\subset X\hat\otimes Y^*$ in the second and the final ones.

Similarly, for $b$ as before, $f\in L^p_c(\R^d;X)$ and $g\in L^{p'}(\R^d;Y^*)$, we have
\begin{equation*}
  \pair{f}{\Pi_{b^*}g}
  =\sum_{k\in\Z}\pair{\D_k f}{b^*\,\E_{k-1}g}
  =\sum_{k\in\Z}\pair{b\,\D_k f}{\E_{k-1}g}
  =\Bpair{b}{\sum_{k\in\Z}\D_{k}f\otimes\E_{k-1} g},
\end{equation*}
where the ultimate right is again an integral pairing of the same type as before.

Hence, if $b\in L^1_{\loc}(\R^d;\bddlin(X,Y))$ and both $f\in L^p_c(\R^d;X)$ and $g\in L^{p'}_c(\R^d;Y^*)$ have finite martingale difference expansions, then
\begin{equation}\label{eq:sumOfPbs}
  \pair{\Lambda_b f}{g}
  :=\pair{(\Pi_b+\Pi_{b^*}^*)f}{g}
  =\Bpair{b}{\sum_{k\in\Z}(\E_{k-1}f\otimes\D_k g+\D_{k}f\otimes\E_{k-1} g)},
\end{equation}
where the sum is finite, and we have an integral pairing between a function taking values in $\bddlin(X,Y)$ and another one with values in $X\otimes Y^*\subset X\hat\otimes Y^*$. We can now elaborate and prove Theorem \ref{thm:parapIntro} as follows:

\begin{theorem}\label{thm:parasum}
Let $X$ and $Y$ be UMD spaces and $b\in\BMO_d(\R^d;\bddlin(X,Y))$. Then the symmetrised paraproduct $\Lambda_b$ defined by \eqref{eq:sumOfPbs} extends to a bounded linear operator from $L^p(\R^d;X)$ to $L^p(\R^d;Y)$ with the norm estimate
\begin{equation*}
\begin{split}
  \Norm{\Lambda_b}{\bddlin(L^p(\R^d;X),L^p(\R^d;Y))}
  &\leq c_d(pp'+\beta_{p,X}\beta_{p,Y})\Norm{b}{\BMO(\R^d;\bddlin(X,Y))} \\
  &\leq c_d'\cdot\beta_{p,X}\beta_{p,Y}\Norm{b}{\BMO(\R^d;\bddlin(X,Y))}
\end{split}
\end{equation*}
where the first $c_d$ is the constant in the $H^1$-$\BMO$ duality \eqref{eq:H1BMOduality}.
\end{theorem}

\begin{proof}
The second inequality follows from the fact that, for any Banach space $E$,
\begin{equation*}
  \beta_{p,E}\geq\beta_{p,\R}=\max(p,p')-1\geq\frac12\max(p,p');
\end{equation*}
see \cite[Proposition 4.2.17(3) and Theorem 4.5.7]{HNVW1} for the first and second steps in the above computation. Thus we concentrate on the first inequality in the statement of the theorem.

By standard density and duality results concerning the $L^p(\R^d;X)$ spaces, it is enough to prove that
\begin{equation*}
   \abs{\pair{\Lambda_b f}{g}} \leq c_d(pp'+\beta_{p,X}\beta_{p,Y})\Norm{b}{\BMO(\R^d;\bddlin(X,Y))},
\end{equation*}
for all $f\in L^p_c(\R^d;X)$ and $g\in L^{p'}_c(\R^d;Y^*)$ with finite martingale difference expansions and norm one in $L^p(\R^d;X)$ and $L^{p'}(\R^d;Y^*)$, respectively.

By \eqref{eq:sumOfPbs} and the $H^1$-$\BMO$ duality \eqref{eq:H1BMOduality} for $E=X\otimes Y^*$ and $E^*\supseteq \bddlin(X,Y)$, we have
\begin{equation*}
\begin{split}
  &\abs{\pair{\Lambda_b f}{g}} 
  \lesssim \Norm{b}{\BMO_d(\R^d;\bddlin(X,Y))}\BNorm{\sum_{k\in\Z}(\E_{k-1}f\otimes\D_k g+\D_{k}f\otimes\E_{k-1} g)}{H^1_d(\R^d;X\hat\otimes Y^*)} \\
  &= \Norm{b}{\BMO_d(\R^d;\bddlin(X,Y))}\BNorm{\sup_{K\in\Z}\Norm{\sum_{k\leq K}(\E_{k-1}f\otimes\D_k g+\D_{k}f\otimes\E_{k-1} g)}{X\hat\otimes Y^*}}{L^1(\R^d)},
\end{split}
\end{equation*}
and the task is reduced to estimating the $H^1_d$ norm on the right.

Since
\begin{equation*}
\begin{split}
  \E_k f &\otimes \E_k g-\E_{k-1}f\otimes\E_{k-1}g \\
  &=(\E_{k-1}+\D_k) f\otimes (\E_{k-1}+\D_k) g-\E_{k-1}f\otimes\E_{k-1}g \\
  &=\E_{k-1}f\otimes\D_k g+\D_{k}f\otimes\E_{k-1} g+\D_k f\otimes \D_k g,
\end{split}
\end{equation*}
we find by telescoping that
\begin{equation*}
  \sum_{k\leq K}(\E_{k-1}f\otimes\D_k g+\D_{k}f\otimes\E_{k-1} g)
  =\E_K f \otimes \E_K g-\sum_{k\leq K}\D_k f\otimes \D_k g.
\end{equation*}
Let us keep in mind that, by the assumptions on $f$ and $g$, all these sums are finite, and hence these functions take their values in the algebraic tensor product $X\otimes Y^*$.

By the triangle inequality, we then have
\begin{equation*}
\begin{split}
  &\BNorm{\sum_{k\in\Z}(\E_{k-1}f\otimes\D_k g+\D_{k}f\otimes\E_{k-1} g)}{H^1(\R^d;X\hat\otimes Y^*)} \\
  &\leq \BNorm{\sup_{K\in\Z}\Norm{\E_K f \otimes \E_K g}{X\hat\otimes Y^*}}{L^1(\R^d)}
    +\BNorm{\sup_{K\in\Z}\Norm{\sum_{k\leq K}\D_k f\otimes \D_k g}{X\hat\otimes Y^*}}{L^1(\R^d)} \\
    &=:I+II.
\end{split}
\end{equation*}
It is immediate that
\begin{equation*}
\begin{split}
  I &\leq \BNorm{\sup_{K\in\Z}\Norm{\E_K f \otimes \E_K g}{X\hat\otimes Y^*}}{L^1(\R^d)} 
  \leq \BNorm{\sup_{K\in\Z}\Norm{\E_K f}{X}\sup_{K\in\Z}\Norm{\E_K g}{Y^*}}{L^1(\R^d)} \\
  &=\Norm{M_df\, M_dg}{L^1(\R^d)}
  \leq\Norm{M_df}{L^p(\R^d)}\Norm{M_dg}{L^{p'}(\R^d)}
  \leq p'\Norm{f}{L^p(\R^d;X)}\cdot p\Norm{g}{L^{p'}(\R^d;Y^*)}
\end{split}
\end{equation*}
by Doob's maximal inequality (see \cite[Theorem 3.2.2]{HNVW1}) in the last step.

On the other hand, introducing independent unbiased random signs $\eps_k$ on some probability space $(\Omega,\mathscr F,\P)$ with expectation $\E_\eps=\int_\Omega(\cdot)\ud\P$, we have
\begin{equation*}
  \E_\eps(\eps_k\eps_j)=\delta_{k,j}
\end{equation*}
and hence
\begin{equation*}
\begin{split}
  \BNorm{\sum_{k\leq K}\D_k f\otimes \D_k g}{X\hat\otimes Y^*}
  &=\BNorm{\E_\eps\Big(\sum_{k\leq K}\eps_k\D_k f\Big)\otimes\Big(\sum_{j\leq K}\eps_j \D_j g\Big)}{X\hat\otimes Y^*} \\
  &\leq\E_\eps\BNorm{\sum_{k\leq K}\eps_k\D_k f}{X}\BNorm{\sum_{j\leq K}\eps_j \D_j g}{Y^*} \\
  &\leq\BNorm{\sum_{k\leq K}\eps_k\D_k f}{L^p(\Omega;X)}\BNorm{\sum_{j\leq K}\eps_j \D_j g}{L^{p'}(\Omega;Y^*)} \\
  &\leq\BNorm{\sum_{k\in \Z}\eps_k\D_k f}{L^p(\Omega;X)}\BNorm{\sum_{j\in \Z}\eps_j \D_j g}{L^{p'}(\Omega;Y^*)}
\end{split}
\end{equation*}
by Kahane's contraction principle (see \cite[Proposition 3.2.10]{HNVW1}) for such random sums in the last step. Thus
\begin{equation*}
\begin{split}
  II &\leq \BNorm{ \Norm{\sum_{k\in \Z}\eps_k\D_k f}{L^p(\Omega;X)}\Norm{\sum_{j\in \Z}\eps_j \D_j g}{L^{p'}(\Omega;Y^*)} }{L^1(\R^d)} \\
  &\leq \BNorm{\sum_{k\in \Z}\eps_k\D_k f}{L^p(\R^d\times \Omega;X)}\BNorm{\sum_{j\in \Z}\eps_j \D_j g}{L^{p'}(\R^d\times\Omega;Y^*)} \\
  &\leq \beta_{p,X}\Norm{f}{L^p(\R^d;X)}\cdot\beta_{p',Y^*}\Norm{g}{L^{p'}(\R^d\times\Omega;Y^*)}
\end{split}
\end{equation*}
by the UMD property of $X$ and $Y^*$ in the last step. We note that $\beta_{p',Y^*}=\beta_{p,Y}$ (see \cite[Proposition 4.2.17(2)]{HNVW1}).

Putting the pieces together, we have completed the proof of Theorem \ref{thm:parasum}.
\end{proof}

\begin{remark}
As discussed in the Introduction, the previous result extends \cite[Proposition 2.2]{HLM}, which contains the case that $p=2$ and $X=Y$ is a Hilbert space, and a weaker statement in the case that $p\in(1,2)\cup(2,\infty)$ and $X=Y$ is a noncommutative $L^p$ space.

Our method of proof is also analogous to, and inspired by, that of \cite[Proposition 2.2]{HLM}. The main new ingredients consist of using the projective tensor product as a replacement of the product in the noncommutative $L^p$ spaces, and the random signs as a replacement of some other noncommutative constructions.
\end{remark}

\section{Set-up for the $T(1)$ theorem}\label{sec:T1setup}

We now turn to a discussion and precise definition of the notions appearing in the statement of our main Theorem \ref{thm:T1intro}.
Let $\mathcal Q(\R^d;X)$ be the space of finite linear combinations of functions of the form $1_Q x$, where $Q\subset\R^d$ is an axes-parallel left-closed, right-open cube and $x\in X$.

Let $t$ be a bilinear form on $\mathcal Q(\R^d;X)\times\mathcal Q(\R^d;X^*)$ associated with a kernel $K\in C(\dot\R^{2d};\bddlin(X))$, where $\dot\R^{2d}=\{(x,y)\in\R^d\times\R^d:x\neq y\}$, in the sense that
\begin{equation}\label{eq:Bfg}
  t(f, g)=\iint \pair{K(x,y)f(y)}{g(x)}\ud y\ud x
\end{equation}
whenever $f\in\mathcal Q(\R^d;X)$ and $g\in\mathcal Q(\R^d;X^*)$ have disjoint supports. (We implicitly assume that the integral on the right makes sense for all such $f,g$; this will follow from the assumptions that will be imposed on $K$ next.) If we can show that
\begin{equation*}
  \abs{t(f,g)}\lesssim\Norm{f}{L^p(\R^d;X)}\Norm{g}{L^{p'}(\R^d;X^*)}
\end{equation*}
for all $(f,g)\in \mathcal Q(\R^d;X)\times\mathcal Q(\R^d;X^*)$, then one obtains the existence of a unique $T\in\bddlin(L^p(\R^d;X),(L^{p'}(\R^d;X^*))^*)$ such that $\pair{Tf}{g}=t(f,g)$ for all $(f,g)\in \mathcal Q(\R^d;X)\times\mathcal Q(\R^d;X^*)$. When $X=X^{**}$ is reflexive (in particular, when $X$ is a UMD space; see  \cite[Theorem 4.3.3]{HNVW1}), it has the so-called Radon--Nikod\'ym property, and $(L^{p'}(\R^d;X^*))^*$ can be identified with $L^p(\R^d;X^{**})=L^p(\R^d;X)$ (see \cite[Theorems 1.3.10 and 1.3.21]{HNVW1}). 

We recall that a family of operators $\mathscr T\subset\bddlin(X)$ is called {\em $R$-bounded} if
\begin{equation}\label{eq:defRbdd}
  \BNorm{\sum_{k=1}^K\eps_k T_k x_k}{L^p(\Omega;X)}\leq C\BNorm{\sum_{k=1}^K\eps_k x_k}{L^p(\Omega;Y)}
\end{equation}
for some (equivalently, by the Khintchine--Kahane inequality \cite[Theorem 3.2.23]{HNVW1}, for all) $p\in[1,\infty)$, for all $K\in\Z_+$, all $x_k\in X$ and all $T_k\in\mathscr T$, where $\eps_k$ are (as before) unbiased random signs on the probability space $(\Omega,\mathcal F,\P)$, and $C<\infty$ depends at most on $p$. The least constant $C$ admissible in \eqref{eq:defRbdd} is denoted by $\mathscr R_p(\mathscr T)$, and $\mathscr R(\mathscr T):=\mathscr R_2(\mathscr T)$ is called the $R$-bound of $\mathscr T$.
See \cite[Chapter 8]{HNVW2} for an extensive treatment of this notion.

We say that $K$ satisfies the {\em $R$-bounded Calder\'on--Zygmund estimates}, if the set
\begin{equation}\label{eq:CZ0}
  \{\abs{x-y}^d K(x,y): (x,y)\in\dot\R^{2d}\}\subset\bddlin(X),
\end{equation}
as well as the sets
\begin{equation}\label{eq:CZ1}
  \Big\{ \abs{x-y}^{d+\delta}\frac{ [K(x,y)-K(z,y)]}{\abs{x-z}^\delta} : x,y,z\in\R^d, \abs{x-y}>2\abs{x-z}\Big\}\subset\bddlin(X)
\end{equation}
and
\begin{equation}\label{eq:CZ2}
  \Big\{ \abs{x-y}^{d+\delta}\frac{ [K(y,x)-K(y,z)]}{\abs{x-z}^\delta} : x,y,z\in\R^d, \abs{x-y}>2\abs{x-z}\Big\}\subset\bddlin(X),
\end{equation}
are $R$-bounded, where $\delta\in(0,1]$.

We also define the action of $t:\mathcal Q(\R^d)\times\mathcal Q(\R^d)\to \mathscr B(X,X^*)\simeq\bddlin(X,X^{**})$ by
\begin{equation*}
  t(\phi,\psi)(x,x^*):=B(\phi\otimes x,\psi\otimes x^*),
\end{equation*}
and we say that $t$ satisfies the {\em weak $R$-boundedness property} if
\begin{equation}\label{eq:WBP}
  \Big\{\frac{t(1_Q,1_Q)}{\abs{Q}}:Q\subset\R^d\text{ cube}\Big\}\subset\bddlin(X,X^{**})
\end{equation}
is $R$-bounded.

Finally, we define $t(1,\cdot)$ by its action on $\mathcal Q_0(\R^d):=\{\psi\in \mathcal Q(\R^d):\int\psi=0\}$ by
\begin{equation*}
    t(1,\psi):=t(\chi,\psi)+\iint\pair{[K(x,y)-K(z,y)](1-\chi(y))}{\psi(x)}\ud x\ud y,
\end{equation*}
where $z\in\supp\psi$ and $\chi\in Q(\R^d)$ is any function that is identically $1$ in a neighbourhood of $\supp\psi$. One routinely checks the convergence of the integral for any such $z$ and $\chi$, and the independence of this definition from their particular choice. We say that $t(1,\cdot)=b\in L_{\loc}^1(\R^d;\bddlin(X))$ if
\begin{equation*}
  t(1,\psi)=\int_{\R^d} b(x)\psi(x)\ud x
\end{equation*}
for all $\psi\in\mathcal Q_0(\R^d)$. Analogously, we define $t(\cdot,1)$ and the meaning of $t(\cdot,1)=b\in L_{\loc}^1(\R^d;\bddlin(X))$.

We can now restate our main Theorem \ref{thm:T1intro} more precisely as follows: (Note in particular that the $t(1,\cdot)$ and $t(\cdot,1)$ defined above provide a rigorous meaning for the heuristic notions of ``$T1$'' and ``$(T^*1)^*$'' featuring in Theorem \ref{thm:T1intro}.)

\begin{theorem}\label{thm:TbSymm}
Let $X$ be a UMD spaces  and $p\in(1,\infty)$.
Let $t$ be a bilinear form on $\mathcal Q(\R^d;X)\times\mathcal Q(\R^d;X^*)$ associated with an $R$-bounded Calder\'on--Zygmund kernel $K$. If moreover that $t$ satisfies the weak $R$-boundedness property and the ``symmetric $T(1)$ assumption'' (with the usual, non-dyadic BMO space!)
\begin{equation}\label{eq:TbSymmAss}
  t(1,\cdot)=t(\cdot,1)=b\in\BMO(\R^d;\bddlin(X)),
\end{equation}
then there is a unique $T\in\bddlin(L^p(\R^d;X))$ such that $t(f,g)=\pair{Tf}{g}$ and
\begin{equation}\label{eq:TbSymmConcl}
  \abs{\pair{Tf}{g}}  \lesssim \beta_{p,X}^2\Big(C_T
    +\Norm{b}{\BMO(\R^d;\bddlin(X))}  \Big) \Norm{f}{p}\Norm{g}{q}
\end{equation}
for all $(f,g)\in \mathcal Q(\R^d;X)\times\mathcal Q(\R^d;Y^*)$. Here $C_T$ is the sum of the $R$-bounds of the collections in \eqref{eq:CZ0} through \eqref{eq:WBP}, and the implicit constant in \eqref{eq:TbSymmConcl} depends at most on $d$, $p$, and $\delta$.
\end{theorem}

\section{Proof of the $T(1)$ theorem}\label{sec:T1proof}

Theorem \ref{thm:TbSymm} is actually a relatively quick corollary of Theorem \ref{thm:parasum} and a recent approach to the operator-valued $T(1)$ theorem from \cite{HH:2016}. In fact, most other dyadic approaches to this theorem would work essentially equally well, but the formulation of some intermediate steps in \cite{HH:2016} is perhaps most convenient for quoting as a black box, avoiding the need of repeating the considerations already covered in previous works, although, unfortunately, it is still necessary to make some technical remarks on the applicability of the results to the setting at hand.

We quote the following result from \cite{HH:2016}:

\begin{theorem}[Operator-valued Dyadic Representation; \cite{HH:2016}, Theorems 1 and 2]\label{thm:dyadRep}
Let the assumptions of Theorem \ref{thm:TbSymm} be in force, except that \eqref{eq:TbSymmAss} is replaced by
\begin{equation}\label{eq:TbNonSymmAss}
  t(1,\cdot)=b_1\in\BMO(\R^d;\bddlin(X)),\qquad   t(\cdot,1)=b_2\in\BMO(\R^d;\bddlin(X)).
\end{equation}
For any $\epsilon\in(0,1)$ and $(f,g)\in \mathcal Q(\R^d;X)\times\mathcal Q(\R^d;X^*)$, we have the representation
\begin{equation*}
\begin{split}
    t(f,g)
    =\E_\omega\Big(C_T\sum_{i,j=0}^\infty 2^{1/\epsilon}2^{-(1-\epsilon)\alpha\max\{i,j\}}\Bpair{S_{\mathscr D^\omega}^{ij}f}{g}
      +\Bpair{(\Pi_{b_1}^{\mathscr D^\omega}+(\Pi_{b_2^*}^{\mathscr D^\omega})^*)f}{g}\Big),
\end{split}
\end{equation*}
where
\begin{itemize}
  \item $\E_\omega$ is the expectation over a random selection of the dyadic system $\mathscr D^\omega$;
  \item $C_T$ is the sum of the $R$-bounds of the collections in \eqref{eq:CZ0} through \eqref{eq:WBP};
  \item each $S_{\mathscr D^\omega}^{ij}$ is an operator with the bound
\begin{equation}\label{eq:opShiftBd}
    \Norm{S_{\mathscr D^\omega}^{ij}}{\bddlin(L^p(\R^d;X))}\lesssim(1+\max\{i,j\})\beta_{p,X}^2;
\end{equation}
  \item $\Pi_{b_1}^{\mathscr D^\omega}$ and $\Pi_{b_2^*}^{\mathscr D^\omega}$ are dyadic paraproducts related to the dyadic system $\mathscr D^\omega$.
\end{itemize}
\end{theorem}

We note that the operators $S_{\mathscr D^\omega}^{ij}$ are so-called {\em operator-valued dyadic shifts} of {\em complexity type} $(i,j)$ associated with the dyadic system $\mathscr D^\omega$, and \eqref{eq:opShiftBd} states the bound for these operators contained in \cite[Theorem 1]{HH:2016}; for the present purposes, the precise form of the $S_{\mathscr D^\omega}^{ij}$ is irrelevant, and we only care about this bound.

In \cite[Theorem 2]{HH:2016}, the Dyadic Representation Theorem \ref{thm:dyadRep} is stated for a different class of test functions (namely, $(f,g)\in(C^1_c(\R^d)\otimes X,C^1_c(\R^d)\otimes X^*)$) and under the qualitative {a priori} assumption that $t(f,g)=\pair{Tf}{g}$ for some $T\in\bddlin(L^p(\R^d;X))$. However, these are only technical auxiliary assumptions used to legitimate the formal manipulations leading to the desired representation. The same representation formula can also be achieved without the {a priori} $L^p(\R^d;X)$-boundedness by using our test functions, and a simple finitary variant of the underlying randomisation process explained in \cite[Section 2.1]{GH:2018}.

We are now ready for:

\begin{proof}[Proof of Theorem \ref{thm:TbSymm}]
The assumptions of Theorem \ref{thm:TbSymm} correspond to those of the Dyadic Representation Theorem \ref{thm:dyadRep} with $b_1=b_2=b$; thus
\begin{equation*}
  \Pi_{b_1}^{\mathscr D^\omega}+(\Pi_{b_2^*}^{\mathscr D^\omega})^*
  =\Pi_{b}^{\mathscr D^\omega}+(\Pi_{b^*}^{\mathscr D^\omega})^*
  =\Lambda_b^{\mathscr D^\omega}
\end{equation*}
is an operator of the form considered in Theorem \ref{thm:parasum}. The conclusion of the Dyadic Representation Theorem \ref{thm:dyadRep} then takes the form
\begin{equation*}
   t(f,g)
    =\E_\omega\Big(C_T\sum_{i,j=0}^\infty 2^{1/\epsilon}2^{-(1-\epsilon)\alpha\max\{i,j\}}\Bpair{S_{\mathscr D^\omega}^{ij}f}{g}
      +\Bpair{\Lambda_{b}^{\mathscr D^\omega}f}{g}\Big).
\end{equation*}
Using the estimate of Theorem \ref{thm:dyadRep} for $S_{\mathscr D^\omega}^{ij}$ and Theorem \ref{thm:parasum} for $\Lambda_{b}^{\mathscr D^\omega}$, we have
\begin{equation*}
\begin{split}
  \abs{t(f,g)}
  &\lesssim\E_\omega\Big(C_T\sum_{i,j=0}^\infty 2^{1/\epsilon}2^{-(1-\epsilon)\alpha\max\{i,j\}}
    (1+\max\{i,j\})\beta_{p,X}^2 \\
    &\qquad\qquad +\beta_{p,X}^2\Norm{b}{\BMO(\R^d;\bddlin(X))}\Big)\Norm{f}{L^p(\R^d;X)}\Norm{g}{L^{p'}(\R^d;X^*)} \\
  &\lesssim\beta_{p,X}^2\Big(C_T+\Norm{b}{\BMO(\R^d;\bddlin(X))}\Big)\Norm{f}{L^p(\R^d;X)}\Norm{g}{L^{p'}(\R^d;X^*)}
\end{split}
\end{equation*}
by summing up a convergent series in the last step.

This is the asserted bound and completes the proof.
\end{proof}

%\bibliography{t1}
%\bibliographystyle{abbrv}

\end{document}